\documentclass[12pt]{amsart}

\usepackage{graphics, amsmath, amsfonts, amssymb, amscd, mathrsfs}
\usepackage[all]{xy}

\newtheorem{theorem}{Theorem}

\newtheorem{lemma}[theorem]{Lemma}
\newtheorem{corollary}[theorem]{Corollary}

\theoremstyle{definition}
\newtheorem{definition}{Definition}

\newtheorem*{nexample}{Example}

\newcommand{\C}{\mathbb{C}}
\newcommand{\D}{\mathbb{D}}
\newcommand{\Z}{\mathbb{Z}}
\renewcommand{\P}{\mathbb{P}}
\renewcommand{\O}{\mathscr{O}}

\newcommand{\N}{\mathbb{N}}

\newcommand{\Aut}{\operatorname{Aut}}
\newcommand{\id}{\operatorname{id}}

\title{Proper holomorphic immersions\\ in homotopy classes of maps from\\ finitely connected planar domains into $\C\times\C^*$}

\author{Finnur L\'arusson} 
\address{School of Mathematical Sciences, University of Adelaide, Adelaide SA 5005, Australia} 
\email{finnur.larusson@adelaide.edu.au}

\author{Tyson Ritter}
\address{School of Mathematical Sciences, University of Adelaide, Adelaide SA 5005, Australia} 
\email{tyson.ritter@adelaide.edu.au}


\date{20 September 2012. Latest minor changes 25 October 2012.}

\thanks{The authors are supported by Australian Research Council grant DP120104110.}

\begin{document}

\begin{abstract}
Gromov, in his seminal 1989 paper on the Oka principle, proved that every continuous map from a Stein manifold into an elliptic manifold is homotopic to a holomorphic map. 
Previously we have shown that, given a continuous map $X \to \C\times\C^*$ from a finitely connected planar domain $X$ without isolated boundary points, a stronger Oka property holds, namely that the map is homotopic to a proper holomorphic embedding. Here we show that every continuous map from a finitely connected planar domain, possibly with punctures, into $\C\times\C^*$ is homotopic to a proper immersion that identifies at most countably many pairs of distinct points, and in most cases, only finitely many pairs. By examining situations in which the immersion is injective, we obtain a strong Oka property for embeddings of some classes of planar domains with isolated boundary points. It is not yet clear whether a strong Oka property for embeddings holds in general when the domain has isolated boundary points. We conclude with some observations on the existence of a null-homotopic proper holomorphic embedding $\C^* \to \C\times\C^*$. 
\end{abstract}

\maketitle
\tableofcontents

\section{Introduction}\label{s:introduction}
\noindent
The problem of which open Riemann surfaces can be embedded into $\C^2$ is a long-standing and difficult unsolved question in complex geometry. (In this paper the term \emph{embedding} will always refer to a holomorphic proper injective immersion.) For details of the progress made towards solving this problem we refer the reader to the introductions of the papers \cite{ForstnericWold2009} and \cite{Ritter2010}. In particular, Globevnik and Stens\o nes \cite{GlobevnikStensones1995} proved that every bounded, finitely connected planar domain in $\C$ without isolated boundary points embeds into $\C^2$. By the Koebe uniformisation theorem \cite[Ch.\ ~V, \textsection ~6, Thm.\ ~2]{Goluzin1969}, every finitely connected planar domain without isolated points in its boundary, other than $\C$ itself, is biholomorphic to such a domain. Moreover, Koebe's theorem shows that such domains are biholomorphic to what we call \emph{circular domains}, namely those domains given as the open unit disc from which a finite number of pairwise disjoint, closed discs of positive radii have been removed. The result of Globevnik and Stens\o nes was extended by Wold \cite{Wold2006}, who introduced a powerful new technique for constructing embeddings of open Riemann surfaces and used it to show that every finitely connected planar domain in $\C$ embeds into $\C^2$.

In \cite{Ritter2010} the second author considered the related embedding problem for open Riemann surfaces when the target space is the 2-dimensional elliptic Stein manifold $\C\times\C^*$. Since the target is no longer contractible, the homotopy class of the embedding becomes of interest. The main result in \cite{Ritter2010} is that, for every circular domain $X$, every homotopy class of continuous maps $X \to \C\times\C^*$ contains an embedding. The author termed this a \emph{strong Oka principle for embeddings of $X$ into $\C\times\C^*$}, as it can be viewed as a special case strengthening of Gromov's Oka principle \cite{Gromov1989}, a deep result that states every homotopy class of continuous maps from a Stein manifold to an elliptic manifold contains a holomorphic map. For further background on the Oka principle, the reader is referred to the survey \cite{ForstnericLarusson2011} and the monograph \cite{Forstneric2011}.

In the current paper we investigate maps from arbitrary finitely connected planar domains into $\C\times\C^*$. Allowing the planar domains to have boundary components that are isolated points makes the problem more difficult since Wold's technique, used in \cite{Ritter2010}, does not give properness at punctures in the domain. We have not yet been able to determine whether a strong Oka principle always holds for embeddings in this more general setting. If, however, we ask for proper holomorphic immersions that are just short of being injective, we are able to prove a generalisation for arbitrary finitely connected planar domains.

In Section~\ref{s:puncturedplanes} we show that every homotopy class of  continuous maps from a finitely punctured plane into $\C\times\C^*$ contains a proper holomorphic immersion that identifies at most either countably many or finitely many pairs of points in the domain, depending on whether the homotopy class is null or not, respectively (Theorem~\ref{t:puncturedplane}). In many cases, we in fact obtain an embedding of the given punctured plane (Theorem~\ref{t:embeddingpuncturedplane}). In Section~\ref{s:planardomains} we apply these results to show that every homotopy class of continuous maps from a circular domain with finitely many punctures into $\C\times\C^*$ contains a proper holomorphic immersion that identifies at most finitely many pairs of points (Theorem~\ref{t:puncturedcirculardomain}). Again, in many circumstances the immersion can be made an embedding, the consequence being that a strong Oka principle holds for embeddings of certain punctured circular domains into $\C\times\C^*$ (Theorem~\ref{t:embeddingpuncturedcirculardomain}, and Corollaries~\ref{c:strongOkapunctureddisc} and \ref{c:strongOkapuncturedcirculardomain}). It also follows, if we disregard the homotopy class of the embedding, that every finitely connected planar domain embeds into $\C\times\C^*$ (Corollary \ref{c:embeddingplanardomain}). The final section contains some partial results on the surprisingly difficult question of the existence of a null-homotopic embedding $\C^* \to \C\times\C^*$.

The current paper, like the second author's earlier paper \cite{Ritter2010}, relies on the fundamental embedding technique introduced by Wold in \cite{Wold2006}. Later papers by Forstneri\v c and Wold \cite{ForstnericWold2009}, and Kutzschebauch, L\o w, and Wold \cite{Kutzschebauch2009} have also been of use to us.

\section{Proper holomorphic immersions of finitely punctured planes}
\label{s:puncturedplanes}

\noindent
We begin by considering the case of a finitely connected planar domain $X\subset \C$ for which each boundary component is an isolated point. Such a domain is equal to the complex plane $\C$ with $n\ge0$ distinct points $a_1,\dots,a_n$ removed. When $n \ge 1$, $X$ is clearly homotopy equivalent to a bouquet of $n$ circles. Since $\C\times\C^*$ is homotopy equivalent to a single circle, if we ignore the trivial case $n=0$, we see that each homotopy class of continuous maps $X\to\C\times\C^*$ is completely determined by the assignment of a winding number $k_j \in \Z$ to each puncture $a_j$, $j = 1,\dots,n$.

\begin{theorem}
\label{t:puncturedplane}
Let $a_1,\dots,a_n \in \C$, $n \ge 0$, be distinct points, and let $X = \C\setminus\{a_1,\dots,a_n\}$. Then:
\begin{enumerate}
\item The null homotopy class of continuous maps $X \to \C\times\C^*$ contains a proper holomorphic immersion that identifies at most countably many pairs of distinct points in $X$.
\item Every non-null homotopy class of continuous maps $X \to \C\times\C^*$ contains a proper holomorphic immersion that identifies at most finitely many pairs of distinct points in $X$.
\end{enumerate}
\end{theorem}

\begin{proof}
In the case $n=0$ the existence of an embedding $X = \C\to\C\times\C^*$ is trivial, so assume that $n\ge 1$. Let the homotopy class of maps $X \to\C\times\C^*$ be determined by the winding numbers $k_j \in \Z$ about the punctures $a_j$, $j = 1,\dots,n$.

Suppose first that we have the null homotopy class, that is, $k_j = 0$ for all $j = 1,\dots,n$. Define
\[
	\psi(z) = \left({(z-c)^{n+1}}/{(z-a_1)\cdots(z-a_n)},e^z\right)\,,
\]	
where $c \in X$ is arbitrary. Then $\psi : X \to \C\times\C^*$ is a proper holomorphic immersion of $X$ into $\C\times\C^*$. Now $\psi$ factors through the null-homotopic map $\C\times\C \to \C\times\C^*$, $(w_1,w_2)\mapsto(w_1,e^{w_2})$, making $\psi$ null-homotopic. We now show that $\psi(x) = \psi(y)$ for at most countably many pairs of distinct points $x, y \in X$.

We write
\[
	f(z) = (z-c)^{n+1}/(z-a_1)\cdots(z-a_n)\,,
\]
so that $\psi(z) = (f(z),e^z)$. The second component of $\psi$ identifies the distinct points $x,y \in X$ if and only if $y = x+2\pi i k$ for some $k \in \Z^* = \Z\setminus\{0\}$. However, for each $k \in\Z^*$, the rational function $f(z) - f(z + 2\pi i k)$ is not identically zero (for otherwise the fibres of $f$ would be infinite, whereas $f$ is not constant), and therefore the equation $f(z) - f(z + 2\pi i k) = 0$ has only finitely many solutions $z \in X$. Considering all $k \in \Z^*$, it follows that $\psi$ identifies at most countably many pairs of distinct points in $X$.

Now suppose that the given homotopy class of maps is non-null. Order the punctures $a_j$ so that for some $r$ satisfying $1 \le r \le n$, we have $k_j \neq 0$ for $j= 1,\dots,r$, and $k_i = 0$ for $i = r+1,\dots,n$. Let
\[
	g(z) = (z-a_1)^{k_1}\cdots(z-a_r)^{k_r}\,.
\]
When $r=n$, the map $\psi(z) = (z, g(z))$ is an embedding of $X$ into $\C\times\C^*$ in the given homotopy class, so assume $r < n$.
If we now let $f = p/q$, where
\[
	q(z) = (z-a_{r+1})\cdots(z-a_n)
\]
and $p\in\C[z]$ is a polynomial with $\deg(p) > \deg(q) \ge 1$ that does not vanish at the zeros of $q$, then $\psi = (f,g) : X \to \C\times\C^*$ is a proper holomorphic map in the given homotopy class. It remains to be shown that with $q$ fixed as above, $p$ can be chosen so that $\psi$ is an immersion that identifies at most finitely many pairs of distinct points in $X$.

We first show that $p$ can be chosen so that $\psi = (f,g) = (p/q,g)$ is an immersion of $X$. Let $z_1,\dots,z_m$ be the critical points of $g$ in $X$. We have $f' = (p'q-pq')/q^2$, and since $q$ does not vanish on $X$ we must ensure that $p'(z_j)q(z_j) \neq p(z_j)q'(z_j)$ for $j = 1,\dots,m$. We begin by choosing any $p \in \C[z]$ with $\deg(p) > \deg(q)$ such that $p'$ is non-zero at each point $z_1,\dots,z_m$. By adding a generic constant to $p$ we can ensure both that $p$ and $q$ have no common zeros and that $f'(z_j)\neq 0$ for all $j=1,\dots,m$, so that $\psi$ is an immersion of $X$.

We now show how to further choose $p$ so that $\psi(x) = \psi(y)$ for at most finitely many pairs of distinct points $x,y \in X$. Begin by considering $f$ and $g$ as holomorphic maps from $\P$ to $\P$. Let $\pi_1, \pi_2 : \P\times\P \to \P$ denote the projection onto the first and second component of $\P\times\P$, respectively.
Form $Y_1$ as the pullback of $f$ by itself,
\[
	\xymatrix{Y_1\ar[d]^{\pi_1} \ar[r]^{\pi_2} &\P\ar[d]^f\\
		\P \ar[r]^f &\P}
\]	    
so that $Y_1$ is the algebraic curve in $\P\times\P$ given by
\[
	Y_1 = \{(x,y) \in \P\times\P : f(x) = f(y)\}\,.
\]
Similarly, let $Y_2$ be the algebraic curve in $\P\times\P$ given as the pullback of $g$ by itself,
\[
	Y_2 = \{(x,y) \in \P\times\P : g(x) = g(y)\}\,.
\]
Note that both $Y_1$ and $Y_2$ share the diagonal $\Delta \subset \P\times\P$ as an irreducible component.

Now suppose that $(f(x), g(x)) = (f(y),g(y))$ for infinitely many pairs of distinct points $x,y \in \P$. This implies that $Y_1\setminus \Delta$ intersects $Y_2\setminus \Delta$ in infinitely many points of $\P\times\P$. By the algebraicity of $Y_1$ and $Y_2$ it follows that the closures $\overline{Y_1\setminus \Delta}$ and $\overline{Y_2 \setminus \Delta}$ in $\P\times\P$ have an irreducible component $Z$ in common.

Consider the projection $\pi_1\rvert_Z : Z\to\P$. Note that $\pi_1\rvert_Z$ cannot be constant, as then $Z = \{a\}\times\P$ for some $a \in \P$, which in turn implies that $f$ and $g$ are constant. It follows that $\pi_1\rvert_Z$ must be surjective, and since $Z\cap\Delta$ is finite we see that for all but finitely many points $x \in \P$ there exists $y\in \P$, $y\neq x$, such that $(f(x),g(x)) = (f(y),g(y))$. By considering $Z \cap (X\times X)$ it follows that for all but finitely many points $x \in X$ there exists $y \in X$, $y\neq x$, such that $(f(x),g(x)) = (f(y),g(y))$.

Let $a$ be a regular value of $g$ such that the fibre $g^{-1}(a)$ is contained in $X$. For ease of exposition, we may also assume that $f(x)$ is a regular value of $f$ for all $x \in g^{-1}(a)$. Now suppose that $f$ is injective on $g^{-1}(a)$ and let $x_1 \in g^{-1}(a)$. Then for every point $x$ sufficiently close to $x_1$, $f$ remains injective on the fibre $g^{-1}(g(x))$. It follows that for infinitely many points $x \in X$ there is no $y \in X$, $y \neq x$, such that $(f(x),g(x)) = (f(y),g(y))$, and therefore that $\psi(x) = \psi(y)$ for at most finitely many pairs of distinct points $x,y \in X$.

Finally, in the event that $f$ is not injective on $g^{-1}(a)$, a small perturbation to the coefficients of $p$ will give the injectivity of $f$ on $g^{-1}(a)$ without destroying any of the existing properties of $f$. To see this, first note that for small perturbations, $p$ continues to be non-vanishing at the zeros of $q$, so $p$ and $q$ still have no common factors. Similarly, it remains true that $f'(z)\neq 0$ at each critical point $z$ of $g$. Since we do not perturb $q$, the map $\psi = (p/q, g)$ continues to be a proper immersion of $X$ into $\C\times\C^*$. Now let $x_1, x_2 \in g^{-1}(a)$, $x_1 \neq x_2$, and suppose $f(x_1) \neq f(x_2)$. This property is clearly preserved for small perturbations of $p$. On the other hand, suppose that $f(x_1) = f(x_2) = c$, so that $x_1$ and $x_2$ are simple zeros of the polynomial $p-cq$, since $a$ was chosen earlier to ensure $c$ is a regular value of $f$. Now, the map from the zeros $(y_1,\dots,y_m)$ of a monic polynomial of degree $m$ to its non-leading coefficients (given by the $m$ elementary symmetric polynomials in $y_1,\dots,y_m$) is continuous, so there exists a small perturbation to the non-leading coefficients of $p-cq$ that perturbs the zero previously at $x_1$ away from that point by a small amount, while keeping $x_2$ and all the other zeros of $p-cq$ fixed. Hence $x_1$ is no longer a zero of $p-cq$ and now $f(x_1) \neq f(x_2)$. Since $\deg p > \deg q$, the desired pertubation to the coefficients of $p-cq$ can be done by perturbing the coefficients of $p$ while keeping $q$ fixed. Repeating this procedure a finite number of times ensures that $f$ is injective on $g^{-1}(a)$.
\end{proof}

Given a proper holomorphic immersion $\psi : X \to \C\times\C^*$ of a finitely punctured plane $X$ provided by Theorem~\ref{t:puncturedplane}, the image $\psi(X)$ is of course a closed analytic subvariety of $\C\times\C^*$ with at most either countably many or finitely many singular points, depending on the homotopy class of $\psi$. It seems likely that Theorem~\ref{t:puncturedplane} could be improved to give $\psi$ such that at each singular point of $\psi(X)$ there are only two locally irreducible components that meet transversally, but we do not pursue this development here.

We are able to produce an embedding in certain homotopy classes of maps $X \to \C\times\C^*$.

\begin{theorem}\label{t:embeddingpuncturedplane}
Let $a_1,\dots,a_n \in \C$, $n \ge 1$, be distinct points, and let $X = \C\setminus\{a_1,\dots,a_n\}$. 
Let a homotopy class of continuous maps $X \to \C\times\C^*$ be given with associated winding numbers $k_j$ about the punctures $a_j$, $j=1,\dots,n$. Suppose one of the following holds.
\begin{enumerate}
\item $k_j \neq 0$ for $j = 1,\dots,n$.
\item $n \ge 2$, with $k_n = 0$, $k_j\neq 0$ for $j = 1,\dots,n-1$, and $k_1 + \cdots+k_{n-1} \neq 0$.
\item $n \ge 3$, with $k_j = 0$ for $j = 3,\dots,n$, $k_1 = 1$, and $k_2 = -1$.
\item $n \ge 3$, with $k_j = 0$ for $j = 2,\dots,n$, and $k_1 = \pm 1$.
\end{enumerate}
Then the given homotopy class contains an embedding $X \to \C\times\C^*$.
\end{theorem}
\begin{proof}
Case~(1) was already covered in the proof of Theorem~\ref{t:puncturedplane} by taking
\[
	\psi(z) = (z, (z-a_1)^{k_1}\cdots(z-a_n)^{k_n})\,.
\]
For case~(2), we take
\[
	\psi(z) = (1/(z-a_n),(z-a_1)^{k_1}\cdots(z-a_{n-1})^{k_{n-1}})\,.
\]
For case~(3), we take 
\[
	\psi(z) = ((z-c)^{n-1}/(z-a_3)\cdots(z-a_{n}), (z-a_1)/(z-a_2))\,,
\]
where $c\in X$ is arbitrary. Finally, in case~(4) we take 
\[
	\psi(z) = (1/(z-a_2)\cdots(z-a_{n}),(z-a_1)^{\pm1})\,.
\]
In each of these cases one component of the map $\psi$ is an injective immersion of $X$, and $\psi$ is also proper, making $\psi$ an embedding with the required winding numbers.
\end{proof}

Given a finitely punctured plane, the general problem of finding an embedding into $\C\times\C^*$ within a given homotopy class separates into two quite different cases, depending on whether the homotopy class is null. For example, note that obviously there is an embedding in each non-null homotopy class of maps $\C^* \to \C\times\C^*$, but the existence or not of a null-homotopic embedding $\C^* \to \C\times\C^*$ is surprisingly difficult to establish. We present some partial results on this topic in Section~\ref{s:nullhomotopicembedding}. More generally, apart from the trivial case of $\C$ we do not know if there exists a finitely punctured plane that null-homotopically embeds into $\C\times\C^*$.

We are however aware of some finitely punctured planes $X$ and non-null homotopy classes of maps $X \to \C\times\C^*$ that contain an embedding with neither of its two components injective. For example, take $X = \C\setminus\{0,a,b\}$, where $a,b\in\C^*$ are distinct. Then $(1/(z-a)(z-b),z^2)$ is an embedding into $\C\times\C^*$, unless $a = -b$, in which case $(1/(z-a)(z+a)^2,z^2)$ is. Given a punctured plane $X$ and a homotopy class of maps $X\to\C\times\C^*$, the punctures with zero winding determine the location of the poles of the first component of any rational embedding in the homotopy class, while the punctures with non-zero winding essentially determine the second component. We therefore have considerable freedom in choosing both the numerator and the order of the poles of the first component, but it is unclear whether this is sufficient to permit in general the construction of an embedding $X \to \C\times\C^*$.

\section{Proper holomorphic immersions of all other finitely connected planar domains}\label{s:planardomains}

\noindent
We now suppose that $X\subset\C$ is a finitely connected planar domain with at least one boundary component not an isolated point. As mentioned in the introduction, by the Koebe uniformisation theorem, we may assume $X$ is the open unit disc from which a finite number of pairwise disjoint closed discs and isolated points have been removed. We call such domains \emph{punctured circular domains}, with the term \emph{circular domain} reserved for the case of no punctures.

\begin{definition}\label{d:puncturedcirculardomain}
A \emph{punctured circular domain} is a domain $X\subset \C$ consisting of the open unit disc $\D$ from which $m\ge0$ closed, pairwise disjoint discs and $n\ge0$ isolated points have been removed. Let the deleted discs have centres $c_i \in \D$ and radii $r_i > 0$, $i=1,\dots,m$, and let the deleted points be $a_1,\dots,a_n \in \D\setminus\bigcup\limits_{i=1}^m(c_i + r_i\overline\D)$. We have
\[
X = \D \setminus (\bigcup_{i=1}^m(c_i + r_i\overline\D) \cup \{a_1,\dots,a_n\})
\]
with constraints $r_i + r_j < \lvert c_i - c_j\rvert$ for $1\le i<j \le m$, and $r_i< 1-\lvert c_i \rvert$ for $i = 1,\dots,m$.
\end{definition}

Let $\pi_1$ and $\pi_2$ denote the projection maps of $\C\times\C^*$ onto the first and second component, respectively. For $r > 0$ define the open annulus
\[
	A_r = \{ z \in \C^* : 1/(r+1) < \lvert z \rvert < (r+1)\}\,.
\]
We call $P_r = r\D \times A_r$ the \emph{cylinder} of radius $r$.

The \emph{nice projection property} for a finite collection of smoothly embedded curves in $\C\times\C^*$ was given in \cite{Ritter2010}, based on a definition introduced in \cite{Kutzschebauch2009}. The definition in \cite{Ritter2010} concerns properties of the curves after they have been projected onto the $\C$-component of $\C\times\C^*$ using $\pi_1$, and in this paper we refer to it as the \emph{$\C$-nice projection property}. If, instead of projecting onto $\C$, we use $\pi_2$ to project the curves onto the $\C^*$-component, we have the following definition.

\begin{definition}\label{definition_niceprojectionproperty}
Let $\gamma_1, \dots,\gamma_m$ be pairwise disjoint, smoothly embedded curves in $\C\times\C^*$, where each $\gamma_i$ has domain either $[0,\infty)$ or $(-\infty, \infty)$. For $i=1,\dots,m$, let $\Gamma_i \subset \C\times\C^*$ be the image of $\gamma_i$ and set $\Gamma = \bigcup\limits_{i=1}^m \Gamma_i$. We say that the collection $\gamma_1,\dots,\gamma_m$ has the \emph{$\C^*$-nice projection property} if there is a holomorphic automorphism $\alpha \in \Aut(\C\times\C^*)$ such that, if $\beta_i = \alpha \circ \gamma_i$ and $\Gamma' = \alpha(\Gamma)$, the following conditions hold.
\begin{enumerate}
\item For every compact subset $K \subset \C^*$ and every $i = 1,\dots,m$, there exists $s > 0$ such that $\pi_2(\beta_i(t)) \notin K$ for $\lvert t \rvert > s$.
\item There exists $M > 0$ such that for all $r \ge M$:
\begin{enumerate}
\item $\C^*\setminus(\pi_2(\Gamma')\cup \overline A_r)$ does not contain any relatively compact connected components.
\item $\pi_2$ is injective on $\Gamma' \setminus \pi_2^{-1}(A_r)$.
\end{enumerate}
\end{enumerate}
\end{definition}

The following result is analogous to the main technical lemma in \cite{Ritter2010} (Lemma~4), adapted to a family of curves with the $\C^*$-nice projection property instead of the $\C$-nice projection property. 

\begin{lemma}\label{l:technical}
Equip $\C\times\C^*$ with a Riemannian distance function $d$. Let $K \subset \C\times\C^*$ be an $\O(\C\times\C^*)$-convex compact set and let $\gamma_1,\dots,\gamma_m$ be pairwise disjoint, smoothly embedded curves in $\C\times\C^*$ satisfying the $\C^*$-nice projection property. Let $\Gamma_i$ be the image of $\gamma_i$, $i=1,\dots,m$, and set $\Gamma = \bigcup\limits_{i=1}^m\Gamma_i$. Suppose that $\Gamma \cap K = \varnothing$. Then, given $r>0$ and $\epsilon > 0$, there exists $\phi \in \Aut(\C\times\C^*)$ such that the following conditions are satisfied.
\begin{enumerate}
\item[(a)] $\sup\limits_{\zeta \in K}d(\phi(\zeta),\zeta) < \epsilon$.
\item[(b)] $\phi(\Gamma) \subset \C\times\C^* \setminus \overline{P}_r$.
\item[(c)] $\phi$ is homotopic to the identity map.
\end{enumerate}
\end{lemma}
\begin{proof}
The proof is similar to that of \cite[Lemma 4]{Ritter2010}, the primary difference being that we exchange the roles of $\C$ and $\C^*$ in the construction of $\phi$. This entails numerous minor modifications. We give the necessary details below, making reference to the proof in \cite{Ritter2010} where expedient.

As in \cite{Ritter2010}, we may assume the automorphism from Definition~\ref{definition_niceprojectionproperty} has already been applied, so that the conditions of the $\C^*$-nice projection property hold directly for the curves $\gamma_i$, $i=1,\dots,m$. The proof that condition~(c) can be ensured is identical to the argument in \cite{Ritter2010}.

Let $K'\subset\C\times\C^*$ be a slightly larger $\O(\C\times\C^*)$-convex compact set containing $K$ in its interior such that $K' \cap \Gamma = \varnothing$ still holds, and shrink $\epsilon$ so that $\epsilon/2 < d(K,\C\times\C^*\setminus K')$. We also assume that $r \ge M$, where $M$ is determined by the $\C^*$-nice projection property for $\gamma_1,\dots,\gamma_m$, and if necessary we take $r$ larger so that $K' \subset \C\times A_r$, and $\gamma_i(0)\in\C\times A_r$ for $i = 1,\dots,m$. Set $\widetilde\Gamma = \Gamma \cap (\C\times\overline{A}_r) = (\pi_2\rvert_\Gamma)^{-1}(\overline{A}_r)$, which is compact by condition~(1) in Definition~\ref{definition_niceprojectionproperty}, and in fact has precisely $m$ connected components $\widetilde\Gamma_1,\dots,\widetilde\Gamma_m$, each $\widetilde\Gamma_i = \Gamma_i \cap (\C\times\overline{A}_r)$ a smoothly embedded compact curve. Following \cite{Ritter2010}, we construct an isotopy of injective holomorphic maps on a neighbourhood of $K' \cup \widetilde\Gamma$ and use the Andersen-Lempert theorem \cite[Thm.\ ~2]{Ritter2010} to obtain $\alpha \in \Aut(\C\times\C^*)$ satisfying:
\begin{enumerate}
\item[(a)] $\sup\limits_{\zeta \in K'}d(\alpha(\zeta),\zeta) < \epsilon/2$.
\item[(b)] $\alpha(\widetilde\Gamma) \subset \C\times\C^*\setminus\overline P_r$.
\end{enumerate}

The automorphism $\alpha$ moves all of $\widetilde\Gamma$ outside of $\overline P_r$, but may move parts of the set $\Gamma\setminus\widetilde\Gamma$ into $\overline P_r$. We let $\Gamma_r = \{\zeta \in \Gamma:\alpha(\zeta)\in\overline P_r\} = \Gamma \cap \alpha^{-1}(\overline P_r)$. By construction, $\pi_2(\Gamma_r) \subset \pi_2(\Gamma) \setminus \overline A_r$. Recall that $r$ was chosen so that $\C^*\setminus(\pi_2(\Gamma)\cup\overline A_r)$ has no relatively compact connected components and such that $\pi_2$ is injective on $\Gamma$ outside of $\C\times A_r$. We construct $\beta \in \Aut(\C\times\C^*)$ that approximates the identity on $\C\times\overline A_r$ and moves $\Gamma\setminus\widetilde\Gamma$ so as to avoid $\alpha^{-1}(\overline P_r)$. The automorphism will have the form $\beta(z,w) = (z+g(w),w)$, where $(z,w) \in \C\times\C^*$ and $g \in \O(\C^*)$. We first construct a continuous map $\widetilde \beta(z,w) = (z +f(w),w)$ from $\C\times(\overline A_r \cup \pi_2(\Gamma))$ to $\C\times\C^*$, where $f : \overline A_r\cup\pi_2(\Gamma) \to \C$ is continuous on $\overline A_r \cup \pi_2(\Gamma)$ and holomorphic on $A_r$, then approximate $f$ uniformly on an appropriate set by $g \in \O(\C^*)$.

Let $s>r$ be chosen so that $\alpha^{-1}(\overline P_r) \subset \overline P_s$. If $(z,w) \in \C\times\C^*$ with $w \notin A_s$, we clearly have $\beta(z,w) \notin \alpha^{-1}(\overline P_r)$ and also $\widetilde\beta(z,w)\notin \alpha^{-1}(\overline P_r)$. We set $f(w) = 0$ for $w \in \overline A_r$, so that $\widetilde \beta = \id$ on $\C\times\overline A_r$. We now show how to define $f$ on $\pi_2(\Gamma) \setminus \overline A_r$ so that $\widetilde \beta$ moves the set $\Gamma' = \Gamma \cap (\overline P_s \setminus(\C\times A_r))$ so as to avoid $\alpha^{-1}(\overline P_r)$.

We now assume each $\gamma_i$ has domain $[0,\infty)$ with $\gamma_i(0) \in \C\times A_r$ by breaking those curves with domains $(-\infty,\infty)$ into two components, as described in \cite{Ritter2010}. Let $k\ge m$ be the new total number of curves $\gamma_i$. For each $\gamma_i(t) = (z_i(t),w_i(t))$ we choose $t^i_0 > 0$ such that $w_i(t^i_0) \in \partial A_r$, with $w_i(t) \in \overline A_r$ for $t < t^i_0$ and $w_i(t) \in \C^* \setminus \overline A_r$ for $t > t^i_0$. Note that $\gamma_i(t^i_0) \notin \alpha^{-1}(\overline P_r)$.

Let $B_i = \sup\{\lvert z_i(t)\rvert : t \ge t^i_0 \mbox{ such that } w_i(t) \in \overline A_s\}$, and let $B = \max\limits_{i = 1,\dots,k}B_i$. Let $L_i = \C\times\{w_i(t^i_0)\}$ and $K_i = L_i \cap \alpha^{-1}(\overline P_r)$. By the injectivity of $\pi_2$ on $\Gamma \setminus (\C\times A_r)$, $L_1,\dots,L_k$, and hence $K_1\dots,K_k$, are all distinct. As $\alpha^{-1}(\overline P_r)$ is $\O(\C\times\C^*)$-convex, each $L_i\setminus K_i$ is connected and unbounded, so for each $i = 1,\dots,k$ we may choose a continuous path $c_i : [0,1] \to L_i\setminus K_i$ satisfying $c_i(0) = \gamma_i(t^i_0)$ and $c_i(1) \in L_i \setminus (T\overline\D\times\{w_i(t^i_0)\})$, where $T > s + \lvert z_i(t^i_0)\rvert + B$.

Now define $\tilde c_i : [0,1] \to \C$ by $\tilde c_i(t) = \pi_1(c_i(t)) - z_i(t^i_0)$, so that $c_i(t) = (z_i(t^i_0) + \tilde c_i(t),w_i(t^i_0))$. We have $\tilde c_i(0) = 0$ and $z_i(t^i_0) + \tilde c_i(1) \notin T\overline\D$. For sufficiently small $\delta > 0$, the curve $(z_i(t^i_0+\delta t) + \tilde c_i(t),w_i(t^i_0 + \delta t))$, $t \in [0,1]$, remains within $\C\times\C^* \setminus \alpha^{-1}(\overline P_r)$ and we still have $z_i(t^i_0 + \delta) + \tilde c_i(1) \notin T\overline\D$.

Define $f : \overline A_r \cup \pi_2(\Gamma) \to \C$ by 
\begin{equation*}
f = \left\{
	\begin{array}{ll}
		0 & \text{on } \overline{A}_r,\\
		\tilde{c}_i(t/\delta) & \text{at } w_i(t^i_0 + t) \text{ for } t \in [0,\delta], i = 1,\dots,k,\\
		\tilde{c}_i(1) & \text{at } w_i(t^i_0 + t) \text{ for } t > \delta, i = 1,\dots,k.
	\end{array} \right.
\end{equation*}

The choice of $T$ made earlier ensures that for all $t \ge t^i_0$ with $w_j(t) \in \overline A_s$ we have
\[
	\lvert z_i(t) + \tilde c_i(1)\rvert \ge \lvert \tilde c_i(1)\rvert - \lvert z_i(t)\rvert > s + B - \lvert z_i(t)\rvert \ge s\,,
\]
and therefore $z_i(t) + \tilde c_i(1) \notin s\D$ for all such $t$.

By the Mergelyan-Bishop theorem \cite{Bishop1958} there exists $g \in \O(\C^*)$ that approximates $f$ uniformly on $(\overline A_r\cup \pi_2(\Gamma))\cap \overline A_s$. By making the approximation sufficiently close we ensure that $\beta(\Gamma)\cap\alpha^{-1}(\overline P_r) = \varnothing$, and $\phi = \alpha \circ \beta$ is then the desired automorphism.
\end{proof} 

By a \emph{bordered Riemann surface} we mean a two-real-dimensional smooth manifold-with-boundary $\overline X$ (not necessarily compact), equipped with a complex structure on the interior compatible with the given smooth structure. Its boundary $\partial \overline X$ is thus a smooth one-dimensional manifold (again, not necessarily compact), namely a disjoint union of circles and lines. In \cite[Thm.~1]{Ritter2010}, a so-called Wold embedding theorem was proved for embeddings of certain bordered Riemann surfaces $\overline X$ into $\C\times\C^*$ such that the image of the boundary components satisfy the $\C$-nice projection property. The corresponding result also holds if the collection of boundary curve images instead satisfy the $\C^*$-nice projection property. In addition, the argument remains valid if we weaken the input to be only a proper holomorphic immersion of $\overline X$ satisfying certain properties as described below.

\begin{theorem}\label{t:Wold}
Let $X$ be an open Riemann surface and $K \subset X$ be a compact set. Suppose that $X$ is the interior of a bordered Riemann surface $\overline X$ whose boundary components are non-compact and finite in number. Let $\psi : \overline X \to \C\times\C^*$ be a proper holomorphic immersion satisfying the following conditions.
\begin{enumerate}
\item $\psi$ identifies at most finitely many pairs of distinct points in $X$.
\item $\psi$ is injective on $\partial \overline X$ and $\psi(\partial \overline X) \cap \psi(X) = \varnothing$.
\item $\psi(\partial\overline X)$ has either the $\C$- or the $\C^*$-nice projection property.
\end{enumerate}
Then there exists a proper holomorphic immersion $\sigma: X \to \C\times\C^*$ that identifies precisely the same pairs of points in $X$ as $\psi$, that approximates $\psi$ uniformly on $K$, and such that $\sigma$ is homotopic to $\psi\rvert_X$. Thus, if $\psi$ is an embedding, then so is $\sigma$.
\end{theorem}
\begin{proof}
If $\psi(\partial \overline X)$ satisfies the $\C$-nice projection property, then by scaling in the $\C^*$-component we may assume that $\overline P_2\cap\partial\overline X = \varnothing$. If, on the other hand, $\psi(\partial \overline X)$ satisfies the $\C^*$-nice projection property, then we apply a translation in the $\C$-component to ensure $\overline P_2\cap\partial\overline X = \varnothing$. Then, in the case that $\psi$ is an embedding of $\overline X$ into $\C\times\C^*$, the proof of Theorem 1 in \cite{Ritter2010} works without modification for both the $\C$- and $\C^*$-nice projection properties, since it only makes use of Lemma~4 in \cite{Ritter2010} and Lemma~\ref{l:technical} in the current paper that hold for each nice projection property, respectively.

Suppose now that $\psi : \overline X \to \C\times\C^*$ is a proper holomorphic immersion satisfying properties~(1)--(3). The image $\psi(X)$ is then a 1-dimensional closed analytic subvariety of $\C\times\C^* \setminus \psi(\partial \overline X)$ with finitely many singular points. The proof of Theorem 5.1 in \cite{ForstnericWold2009}, which follows an argument detailed in \cite[Prop.~3.1]{Wold2006a}, continues to hold even in the case when $X$ has punctures, by virtue of the properness of $\psi$. That is, $\psi(X)$ admits an exhaustion by $\O(\C\times\C^*)$-convex compact sets satisfying the conclusion of Lemma 5 in \cite{Ritter2010}. The construction given in the proof of Theorem 1 in \cite{Ritter2010} therefore continues to work without modification, the result being a proper holomorphic immersion $X \to \C\times\C^*$ that identifies the same pairs of points in $X$ as the given map $\psi$.

The proof that $\sigma$ is homotopic to $\psi$ remains unchanged from \cite[Lemma 6]{Ritter2010}.
\end{proof}

Let $X$ be a punctured circular domain. In \cite{Ritter2010} it was shown that when $X$ has no punctures, every homotopy class of maps $X \to \C\times\C^*$ contains an embedding, so we assume that $X$ has at least one puncture. Using the proper immersions constructed in Section~\ref{s:puncturedplanes} we obtain the following main result of this section.

\begin{theorem}\label{t:puncturedcirculardomain}
Let $X = \D \setminus (\bigcup\limits_{i=1}^m(c_i + r_i\overline\D) \cup \{a_1,\dots,a_n\})$, $m\ge0$, $n\ge 1$, be a punctured circular domain. Then every homotopy class of continuous maps $X \to \C\times\C^*$ contains a proper holomorphic immersion that identifies at most finitely many pairs of distinct points in $X$.
\end{theorem}
\begin{proof}
As in the case of a finitely punctured plane, $X$ is homotopy equivalent to a bouquet of $m+n$ circles, and therefore the homotopy class of a continuous map $X \to \C\times\C^*$ is completely determined by the winding $k_j\in\Z$ about each puncture $a_j$, $j=1,\dots,n$, together with the winding $s_i\in\Z$ about each hole $c_i + r_i\overline\D$, $i = 1,\dots,m$. For $i = 1,\dots,m$, let $\gamma_i = c_i + r_i\partial\D$, and let $b_i = c_i + r_i\sqrt{-1}\in \gamma_i$. We also set $\gamma_0 = \partial \D$ and $b_0 = \sqrt{-1} \in \gamma_0$.

Define the punctured plane $Y = \C\setminus\{a_1,\dots,a_n,b_0,b_1,\dots,b_m,c_1,\dots,c_m\}\supset X$, and consider the homotopy class of maps $Y \to \C\times\C^*$ with winding $k_j$ at each puncture $a_j$, $j = 1,\dots,n$, winding $-1$ at each puncture $b_0,b_1,\dots,b_m$, and winding $s_i + 1$ at each puncture $c_i$, $i = 1,\dots,m$. As there is non-zero winding at $b_0$, Theorem~\ref{t:puncturedplane} gives a proper holomorphic immersion $\psi = (f,g) : Y \to\C\times\C^*$ with rational components that identifies at most finitely many pairs of distinct points.

Define the bordered Riemann surface $\overline X = X \cup \bigcup\limits_{i=0}^m (\gamma_i \setminus \{b_i\})$ by taking $X$ together with each of its boundary curves $\gamma_i$, with the distinguished points $b_i$ removed. Note that each boundary component of $\overline X$ is non-compact. The map $\psi\rvert_{\overline X}$ gives a proper holomorphic immersion of $\overline X$ into $\C\times\C^*$ that identifies at most finitely many pairs of distinct points in $X$. Moreover, $\psi\rvert_{X}$ is in the given homotopy class. Assuming for the moment that conditions~(2) and (3) in Theorem~\ref{t:Wold} hold for $\psi\rvert_{\overline X}$, the theorem gives a proper holomorphic immersion $\sigma : X \to \C\times\C^*$ in the given homotopy class that identifies precisely the same pairs of points in $X$ as $\psi$.

To ensure that condition~(3) holds, first note that $g$ has a pole of order $1$ at each of the points $b_0,\dots,b_m$, so $g(z) \to \infty$ as $z \to b_i$ along each curve $\gamma_i$. For $i = 0,\dots,m$, let
\[
	\theta_i = \lim_{\substack{z \in \gamma_i\\z\to b_i}}\arg g(z)\mod \pi\,.
\]
Then $\theta_i \in [0,\pi)$ is well defined and, defining $\widetilde g_i(z) = g(z)(z-b_i)$, we in fact have 
\[
	\theta_i = \arg \widetilde g_i(b_i)\mod \pi
\] by the choice of $b_i \in \gamma_i$. If $\theta_i \neq \theta_j$ for all $0\le i < j \le n$, then, assuming the images $\psi(\gamma_i)$ are pairwise disjoint, $\psi(\partial \overline X)$ satisfies the $\C^*$-nice projection property. Otherwise, choose a point $d \in Y$ far away from $\D$, and add an additional puncture in $Y$ at $d$ with winding $1$. This has the effect of introducing a zero of order $1$ at $d$ into $g$. If $\vert d\rvert$ is sufficiently large, we will still have $\theta_i \neq \theta_j$ for any pairs $i,j$ where this already held. Now suppose that we previously had $\theta_i = \theta_j$ for some pair $i\neq j$. Provided that $d$ is chosen to lie off the real line containing $b_i$ and $b_j$, equality will no longer hold for the modified $g$. Thus, for a generic choice of large $d$, $\psi(\partial \overline X)$ now satisfies the $\C^*$-nice projection property.

It remains to be shown that we can ensure that condition~(2) in Theorem~\ref{t:Wold} holds. The only way in which condition~(2) may fail is if for at least one point $x \in \partial\overline X$ there exists $y \in \overline X$, $y\neq x$, such that $\psi(y) = \psi(x)$. Note that there are at most finitely many points $x \in \C$ with this property. Suppose there is a single such point $x \in \partial\overline X$ (the case when several such points exist is handled by a straightforward generalisation of the following argument). Let $(x_j)_{j\in\N}$ be a sequence of points in $X$ converging to $x$. For each $j \in \N$, let $v_j \in \O(\C)$ be the unique polynomial with a simple zero at each puncture of $Y$ and no other zeros, such that $v_j(x_j) = x-x_j$. Fix a compact set $K$ containing $\overline \D$ in its interior, and let $\epsilon > 0$. Then $\lVert v_j \rVert_K < \epsilon$ for sufficiently large $j \in \N$. Let $\rho(z) = z + v_j(z)$ for some such $j$. For $\epsilon$ sufficiently small, $\rho$ is injective on $\overline \D$ and hence restricts to a biholomorphism of $\overline X$ onto $\rho(\overline X)$. Thus $\rho(\overline X)$ is a bordered Riemann surface such that $x \notin \rho(\partial\overline X) = \partial (\rho(\overline X))$, since $\rho(x_j) = x$ and $x_j \in X$. For small $\epsilon$, $\rho(\partial\overline X)$ will not contain any of the finitely many points $y \in \C$ for which there exists $z\neq y$ satisfying $\psi(z) = \psi(y)$, and the $\C^*$-nice projection property will still be satisfied for $\psi(\rho(\partial\overline X))$. Thus the conditions of Theorem~\ref{t:Wold} are satisfied for $\overline X$ with the map $\psi \circ \rho$.
\end{proof}

As for finitely punctured planes, for many choices of a punctured circular domain $X$ and a homotopy class of maps $X \to \C\times\C^*$ we are able to obtain an embedding in the given homotopy class.

\begin{theorem}\label{t:embeddingpuncturedcirculardomain}
Let $X = \D \setminus (\bigcup\limits_{i=1}^m(c_i + r_i\overline\D) \cup \{a_1,\dots,a_n\})$, $m\ge0$, $n\ge 1$, be a punctured circular domain, and suppose a homotopy class of maps $X \to \C\times\C^*$ is given with winding $k_j$ about each puncture $a_j$, $j=1,\dots,n$, and winding $s_i$ about each hole $c_i + r_i \overline \D$, $i=1,\dots,m$. Suppose that one of the following holds.
\begin{enumerate}
\item $k_j \neq 0$ for $j = 1,\dots,n$.
\item $k_j \neq 0$ for $j = 1,\dots,n-1$, and $k_n = 0$.
\item $n\ge3$, $k_1 = 1$, $k_2 = -1$, $k_j = 0$ for $j = 3,\dots,n$, and $s_i = 0$ for $i=1,\dots,m$.
\item $n\ge3$, $k_1 = \pm1$, $k_j = 0$ for $j = 2,\dots,n$, and $s_i = 0$ for $i=1,\dots,m$.
\item $n\ge2$, $k_j = 0$ for $j = 1,\dots,n$, and $s_i = 0$ for $i = 1,\dots,m$.
\item $n\ge2$, $m = 1$, $k_j = 0$ for $j = 1,\dots,n$, and $s_1 = \pm1$.
\end{enumerate}
Then the given homotopy class contains an embedding $X \to \C\times\C^*$.
\end{theorem}
\begin{proof}
Let $Y = \C\setminus\{a_1,\dots,a_n,b_0,b_1,\dots,b_m,c_1,\dots,c_m\}$, where the points $b_i$ are defined as in the proof of Theorem~\ref{t:puncturedcirculardomain}. We also define $\overline X$ as in the proof of Theorem~\ref{t:puncturedcirculardomain}.

In case~(1), Theorem~\ref{t:embeddingpuncturedplane} gives an embedding $\psi = (\id\vert_Y,g) : Y \to \C\times\C^*$ with winding $k_j\neq 0$ at each $a_j$, $j=1,\dots,n$, winding $-1$ at $b_0,\dots,b_m$, and winding $s_i+1$ at each $c_i$, $i=1,\dots,m$. By the same procedure as in the proof of Theorem~\ref{t:puncturedcirculardomain} we can modify $g$ to ensure that $\psi(\partial \overline X)$ satisfies the $\C^*$-nice projection property. Applying Theorem~\ref{t:Wold} to $\psi\rvert_{\overline X}$ gives the desired embedding $X \to \C\times\C^*$. Case~(2) follows analogously.

For case~(3), let $\psi : \overline X \to \C\times\C^*$ be given by
\[
	\psi(z) = ((z-d)/(z-a_3)\cdots(z-a_{n})(z-b_0)\cdots(z-b_m), {(z-a_1)}/{(z-a_2)})\,,
\]
where $d \in \C\setminus \overline \D$. The second component of $\psi$ is an injective immersion, and $\psi$ is an embedding of the bordered Riemann surface $\overline X$ in the given homotopy class. As in the proof of Theorem~\ref{t:puncturedcirculardomain}, we choose $d$ to lie off the real lines joining each pair $b_i\neq b_j$, and with $\lvert d \rvert$ large. Then $\psi(\partial \overline X)$ has the $\C$-nice projection property, and Theorem~\ref{t:Wold} gives the desired embedding $X \to \C\times\C^*$. Case~(4) is handled similarly, taking
\[
	\psi(z) = ((z-d)/(z-a_2)\cdots(z-a_{n})(z-b_0)\cdots(z-b_m), {(z-a_1)}^{\pm1})\,,
\]
while in case~(5) we take
\[
	\psi(z) = ((z-d)/(z-a_1)\cdots(z-a_n)(z-b_0)\cdots(z-b_m),z-2)\,.
\]
In each case, a generic choice of large $d$ ensures that $\psi(\partial \overline X)$ satisfies the $\C$-nice projection property. 

Finally, in case~(6), let
\[
	\psi(z) = (1/(z-a_1)\cdots(z-a_n), (z-b_1)^{\pm1}(z-b_0)^{\mp1})\,.
\]
Then $\psi:\overline X \to \C\times\C^*$ is an embedding in the given homotopy class, and $\psi(\partial \overline X)$ satisfies the $\C^*$-nice projection property, with one boundary curve going to $0$ while the other goes to $\infty$. 
\end{proof}

If we temporarily ignore the homotopy classes of our embeddings, we have the following immediate corollary of Theorems \ref{t:embeddingpuncturedplane} and \ref{t:embeddingpuncturedcirculardomain}.

\begin{corollary}\label{c:embeddingplanardomain}
Every finitely connected planar domain embeds into $\C\times\C^*$.
\end{corollary}

Using the procedure in the proof of Theorem~\ref{t:Wold} it is clear that a general solution to the embedding problem for finitely punctured planes would give a strong Oka principle for embeddings of every finitely connected planar domain into $\C\times\C^*$. While such a result is currently out of reach, we have the following corollaries of Theorem~\ref{t:embeddingpuncturedcirculardomain} giving new situations in which a strong Oka principle for embeddings does indeed hold.

\begin{corollary}\label{c:strongOkapunctureddisc}
Let $X$ be an open disc with either one or two punctures. Then every homotopy class of continuous maps $ X \to \C\times\C^*$ contains an embedding.
\end{corollary}

\begin{corollary}\label{c:strongOkapuncturedcirculardomain}
Let $X$ be a punctured circular domain with a single puncture. Then every homotopy class of continuous maps $X \to \C\times\C^*$ contains an embedding.
\end{corollary}

\section{On a null-homotopic embedding of $\C^*$ into $\C\times\C^*$}\label{s:nullhomotopicembedding}

\noindent
As discussed in Section 2, while it is trivial to find an embedding $\C^* \to \C\times\C^*$ in any non-null homotopy class, the question on the existence of a null-homotopic embedding $\C^* \to \C\times\C^*$ is surprisingly difficult. We present here some partial results towards answering this question.

We begin by presenting a few explicitly defined maps that have some, but not all, of the properties of a null-homotopic embedding of $\C^*$ into $\C\times\C^*$. Note that a map $\C^*\to\C\times\C^*$ is null-homotopic if and only if it factors through the map $\id\times\exp:\C\times\C\to\C\times\C^*$, that is, if it has the form $(f,e^g)$ for some $f,g\in\O(\C^*)$.

\begin{nexample}  
(a)  The map $\C^*\to\C\times\C^*$, $z\mapsto (z+1/z, e^{\pi iz})$, is a null-homotopic proper immersion that identifies the points
\[ -k -\sqrt{k^2+1} \quad\textrm{and}\quad k -\sqrt{k^2+1} \]
and the points
\[ -k +\sqrt{k^2+1} \quad\textrm{and}\quad k +\sqrt{k^2+1} \]
for each $k\in\mathbb N$, and is otherwise injective.  It induces a null-homotopic embedding of $\C^*$ into $\C\times\C^*$ with each point of an infinite discrete set blown up.

(b)  The map $\C^*\to\C\times\C^*$, $z\mapsto (e^z, e^{iz})$, is a null-homotopic injective immersion that fails to be proper both at $0$ and $\infty$.  The same formula defines an embedding of $\C$ into $\C^*\times\C^*$.

(c)  The map $\C^*\to\C\times\C^*$, $z\mapsto (z, e^{1/z})$, is a null-homotopic injective immersion that is proper at $\infty$ but not at $0$.  The same formula defines an embedding of $\C^*$ into $\C^*\times\C^*$.
\end{nexample}

The following results show that the first component of a null-homotopic holomorphic injection of $\C^*$ into $\C\times\C^*$ cannot have any symmetries and cannot be proper, that is, cannot have a pole at both $0$ and $\infty$.

\begin{theorem}  \label{t:no-symmetries}
Let $(f,e^g):\C^*\to\C\times\C^*$ be a holomorphic injection.  If $\sigma:\C^*\to\C^*$ is holomorphic and $f\circ\sigma=f$, then $\sigma=\mathrm{id_{\C^*}}$.
\end{theorem}

\begin{proof}
If $z\in\C^*$, then $f(\sigma(z))=f(z)$, so if $\sigma(z)\neq z$, then $e^{g(\sigma(z))}\neq e^{g(z)}$, so $g(\sigma(z))-g(z)\in\C\setminus 2\pi i\Z$.  Hence the holomorphic function $z\mapsto g(\sigma(z))-g(z)$ on $\C^*$ takes values in $\C\setminus 2\pi i\Z^*$, so it is constant, say $c$.  If $c=0$, then $\sigma=\mathrm{id_{\C^*}}$.  So suppose $c\in\C\setminus 2\pi i\Z$.  Then $\sigma$ has no fixed points.

If $\sigma(z)=\sigma(w)$, then 
\[ f(z)=f(\sigma(z))=f(\sigma(w))=f(w) \] 
and 
\[ g(z)+c=g(\sigma(z))=g(\sigma(w))=g(w)+c, \] 
so $g(z)=g(w)$, and $z=w$.  Thus $\sigma$ is injective.  Since $\sigma$ is also surjective by Picard's little theorem, $\sigma$ is an automorphism of $\C^*$.  Since $\sigma$ has no fixed points, $\sigma(z)=az$ for some $a\in\C^*$, $a\neq 1$.

If $\lvert a\rvert\neq 1$, then $f$ is constant by Liouville's theorem.  If $\vert a \rvert=1$ and $a$ is not a root of unity, so $\{a^n:n\in\Z\}$ is dense in the circle, then $f$ is constant by the identity theorem.  But then $e^g:\C^*\to\C^*$ is injective, which is absurd.

It follows that $a$ is a $k$-th root of unity for some $k\geq 1$, so
\[ g(z)=g(a^k z)=g(\sigma^k(z))=g(z)+kc \]
and $c=0$, contradicting the assumption that $c\in\C\setminus 2\pi i\Z$.  Thus $\sigma=\mathrm{id_{\C^*}}$.
\end{proof}

\begin{theorem}  \label{t:not-rational}
Let $(f,e^g):\C^*\to\C\times\C^*$ be a holomorphic injection.  Then $f$ is not proper.
\end{theorem}

\begin{proof}
Since $e^g$ is not injective, $f$ is not constant.  Let $X$ be the pullback of $f$ by itself, that is, the $1$-dimensional subvariety of $\C^*\times\C^*$ defined by the equation $f(x)=f(y)$.  Consider the holomorphic function $G:X\to\C\setminus{2\pi i\Z^*}$, $(x,y)\mapsto g(x)-g(y)$.  It is zero only on the diagonal $\Delta\subset X$.

Suppose $f$ is proper.  Then $f$ is rational, so $X$ is an affine algebraic curve and $G$ is constant on each connected component of $X$.  Hence, $\Delta$ is a connected component of $X$.  We have $X\neq\Delta$, for otherwise $f$ would be injective and therefore could not have a pole at both $0$ and $\infty$.

Let $Y$ be a connected component of $X\setminus\Delta$.  Since $Y\cap\Delta=\varnothing$, $Y$ is not the line $\C^*\times\{a\}$ or $\{a\}\times\C^*$ for any $a\in\C^*$, so each of the two projections $Y\to\C^*$ has cofinite image.  Thus there are points $(x,y)$ and $(x',y')$ in $Y$ with $y=x'$.  Then $(x,y')\in X$ and $G(x,y')=G(x,y)+G(x',y')$.  This shows that the finite set $G(X\setminus\Delta)\subset\C^*$ is closed under multiplication by $2$, which is absurd.
\end{proof}

The following corollary is immediate.

\begin{corollary}  \label{c:essential-singularities-inevitable}
Let $(f,e^g):\C^*\to\C\times\C^*$ be a proper holomorphic injection.  Then $f$ has an essential singularity at $0$ or $\infty$.
\end{corollary}

\end{document}